\documentclass[
final
, nomarks
]{dmtcs-episciences}

\usepackage[utf8]{inputenc}
\usepackage{subfigure}

\usepackage[round]{natbib}

\usepackage{amssymb}
\usepackage{amsmath}
\usepackage{amsthm}
\usepackage{mathtools}

\newcommand\ds{\displaystyle}                    
\newcommand\ts{\textstyle}                       
\newcommand\mb[1]{\mathbb{#1}}                   
\newcommand\mc[1]{\mathcal{#1}}                  
\newcommand\tendsto[1]{\xrightarrow[#1]{ }}      
\newcommand\ff[1]{^{\underline{#1}}}             
\newcommand\rf[1]{^{\overline{#1}}}              
\newcommand\ol[1]{\overline{#1}}                 
\newcommand\Pb{P}                       
\newcommand\Ex{E}                       
\newcommand\Va{V}                       
\newcommand\Beta{B}                     
\DeclarePairedDelimiter\lr{\lparen}{\rparen}     
\DeclarePairedDelimiter\lrb{\lbrack}{\rbrack}    
\DeclarePairedDelimiter\abs{\lvert}{\rvert}      
\DeclarePairedDelimiter\cl{\lceil}{\rceil}       
\DeclarePairedDelimiter\fl{\lfloor}{\rfloor}     

\newtheorem{theorem}{Theorem}
\newtheorem{corollary}[theorem]{Corollary}
\newtheorem{lemma}[theorem]{Lemma}
\theoremstyle{remark}
\newtheorem{remark}{Remark}

\allowdisplaybreaks

\author{Kevin Durant
  \and Stephan Wagner\thanks{Supported by the National Research Foundation of South Africa, grant 96236.}}
\title{On the centroid of increasing trees}
\affiliation{
  Department of Mathematical Sciences, Stellenbosch University, South Africa
}
\received{2018-2-27}
\revised{2018-11-29}
\accepted{2019-8-3}
\begin{document}
\publicationdetails{21}{2019}{4}{8}{4325}
\maketitle
\begin{abstract}
A centroid node in a tree is a node for which the sum of the distances to all other nodes attains its minimum, or
equivalently a node with the property that none of its branches contains more than half of the other nodes.
We generalise some known results regarding the behaviour of centroid
nodes in random recursive trees (due to Moon) to the class of
very simple increasing trees, which also includes the families of plane-oriented
and $d$-ary increasing trees. In particular, we derive limits of distributions
and moments for the depth and label of the centroid node nearest to the root, as
well as for the size of the subtree rooted at this node.
\end{abstract}

\section{Introduction}
\label{sec:intro}

Increasing trees are rooted, labelled trees in which paths away from the root
are labelled in increasing order. Apart from their practical relevance to
problems in computer science (see \citet{bergeron92varieties} for a useful
summary), they possess a large structural variety that is especially interesting
from a combinatorial point of view. This variety stems from a weighting scheme
in which a tree's nodes are assigned weights according to their out-degrees
(that is, their number of children) and the weight of the tree is defined to be
the product of those of its nodes. Each set of weights gives rise to a different
family of trees,~$\mc{T}$, which may emphasise or suppress certain structural
features within the class; say by restricting the possible out-degrees of nodes,
or by assigning significance to the order of their children.

Given such a family of trees, one can pose combinatorial questions in terms of
the subset $\mc{T}_n \subset \mc{T}$ of trees of size~$n$ (that is, the set of
trees made up of $n$~nodes), from which trees are drawn at random according to
their relative weights. For example, one might be interested in the sum of the
weights of the trees in this set~\citep{bergeron92varieties}, the height and
profile of a typical tree~\citep[Chapter~6]{drmota09random}, or the asymptotic
distribution of the number of subtrees of a fixed size over all elements
of~$\mc{T}_n$~\citep{fuchs12limit}.

Our interest here is in centroid nodes of large, random increasing trees. We say
that a node~$v$ in a graph is a \emph{centroid} if it minimises the average
distance to any other node in the graph (as opposed to a \emph{central point} or \emph{centre},
which minimises the maximum). An equivalent definition for trees can be given in terms of branches.
The \emph{branches} of a tree~$T$ at node~$v$ are the maximal
subtrees of~$T$ that do not contain~$v$. We will also use the shorthand
``branches of~$T$'' for the branches of~$T$ at its root node. The members of the
branch of~$v$ containing the root are \emph{ancestral} nodes, while members of
the remaining branches are \emph{descendants}.
It is well known that $v$ is a centroid if each of the tree's branches
at~$v$ contains at most half of the tree's nodes~\citep{zelinka68medians}.

Furthermore, every tree has either one or
two centroids---however the latter case only occurs when the size~$n$ of the
tree is even, the centroid nodes are adjacent, and the largest branch of each
has exactly $n/2$ nodes (essentially, when there is a centroid
edge), see~\cite{jordan69assemblages}.

Specifically, the probability of a random increasing tree having two centroids
vanishes at a rate of~$\Theta(1/n)$ as $n$ tends to infinity (see
Lemma~\ref{lem:E_n(U_m)}), so it seems reasonable to focus solely on parameters
that involve the first centroid---i.e., the one closest to the tree's root. All
of the parameters we consider here---depth, label, and number of
descendants---lead to random variables over the set of trees of size~$n$, and as
$n \to \infty$, these random variables converge to limits that can be described
in full.

Before we can summarise these or any other known results properly, we must
highlight the subclass of \emph{very simple increasing trees}, which are
particularly interesting (from a combinatorial point of view, at least) because
they can be described and characterised using a straightforward growth process:
begin with the root node~$1$, and repeatedly attach nodes to the tree according
to certain probabilistic rules, determined by the family's out-degree weights.
The simplest such family is that of recursive trees, in which trees are formed
by attaching new nodes uniformly at random to existing ones.

\subsection{New and existing results}
\label{sec:intro results}

For the family of recursive trees, \citet{moon02centroid} has presented a number
of noteworthy results that detail both the depth and label of the centroid lying
nearest to the root node. Similar results on simply generated trees can be found in 
two papers by Meir and Moon dedicated to the topic of centroid nodes
\citep{meir02centroid,moon85expected}. Letting $M =\fl{(n-1)/2}$, it turns out
that the average depth of the centroid in a
recursive tree of size~$n$ is $M/(n-M)$, and the expected label of the centroid
is:
\begin{equation}\label{eqn:moon label}
  \frac12 + \frac{n(n+1)}{2(n-M)(n+1-M)}.
\end{equation}
It follows that as $n$ tends to infinity, the mean depth and label of the
nearest centroid approach $1$ and~$5/2$ respectively---an early indication of
the strong correlation between the root and centroid nodes of an increasing
tree.

Indeed, Moon also showed that the probability that the centroid and root
coincide tends to $1-\ln2 \approx 0.31$ as $n \to \infty$ (along with the
limiting probability for an arbitrary node; see Corollary~\ref{cor:P(L = k)}
below), and that the probability of the centroid having depth at least~$h$ tends
to $(\ln2)^h/h!$. Furthermore, it was shown that the probability that the
ancestral branch of the centroid has $B$~nodes is:
\begin{equation}\label{eqn:moon subtree}
  \frac{n}{(n-B)(n-B+1)}\lr*{1 - \sum_{b=\cl{(n+1)/2}}^{\fl{n-B-1}} 1/h},
\end{equation}
and that the mean of the proportion of the tree accounted for by this branch
approaches $(\ln2)^2/2 \approx 0.24$.

Although exact finite expressions such as~\eqref{eqn:moon label}
and~\eqref{eqn:moon subtree} do not seem to be within reach for the broader
class of very simple increasing trees, we will show here that the behaviour of
the centroid can be described both precisely and in some generality in the limit
$n \to \infty$. In particular, our theorems will be formulated in terms of
limiting distributions and limits of moments, and all of the above (asymptotic)
results will follow from them as special cases. This focus on asymptotic
behaviour and full limiting distributions is one of two notable differences
between our work and that of Moon, the other being methodological: whereas the
above results were obtained mostly using elementary counting arguments, our
exposition relies on generating functions, singularity analysis, and a variety
of other tools from the field of analytic combinatorics.

Briefly, our results are as follows: the depth of the nearest centroid in a
large very simple increasing tree converges to a discrete limiting distribution
that is concentrated around the root, i.e., one that decreases exponentially for
larger integer values. The associated moments converge as well, and we find, for
example, that the expected depth of the centroid in a random plane-oriented or
binary increasing tree tends to $1/2$ and~$2$ respectively.

Similarly, the label of the centroid converges to a discrete limiting
distribution that favours values close to~$1$. The mean label of the centroid in
plane-oriented and binary trees tends to $7/4$ and~$4$, respectively, and it
will follow from the limiting distribution of the centroid's label that the
probability of the root and centroid coinciding tends to $0.59$ and~$0$ in these
two families. In all of these results it is noticeable that the root is further
from the centroid in a binary increasing family than in any other type of
increasing tree.

Lastly, we show that the limiting distribution for the proportion of the tree
accounted for by the centroid's ancestral branch is a combination of a point
measure at~$1$ and a decreasing density on $[1/2, 1)$. Its expected value
approaches (roughly) $0.13$ and~$0.38$ in plane-oriented and binary increasing
trees respectively. Several of these results are partially recognisable in
Figure~\ref{fig:recursive}.

The remainder of this paper is organised as follows: we describe families of
very simple increasing trees in more detail in Section~\ref{sec:trees}, and then
address the depth, label, and subtree size of the centroid in
Sections~\ref{sec:depth}, \ref{sec:label}, and~\ref{sec:root branch},
respectively.

As a final introductory note, we should mention that there are several
interesting results involving centroids of another, similar class of random
trees---the so-called \emph{simply generated} families. Like increasing trees
these are rooted, weighted trees, but without the additional restriction on
paths leading away from the root. By design, both plane (Catalan) and labelled
(Cayley) trees can be seen as simply generated families.

In particular, it is known that the centroid of a large simply generated tree
typically has exactly three branches of
size~$\Theta(n)$~\citep[Theorem~4]{aldous94recursive}, and that in the limit $n
\to \infty$ each of these branches itself behaves like a scaled, random simply
generated tree. Furthermore, the proportion of the tree accounted for by the
centroid's ancestral branch approches~$0.41$, while the proportions for the
remaining two branches tend to~$0.44$ and~$0.15$~\citep{meir02centroid}.
Remarkably, these results apply to the class of simply generated trees as a
whole, because all simply generated families share the same limiting object: the
continuum random tree~\citep{aldous91continuum2}.

On the other hand, there are also a number of results that describe the
behaviour of a complementary centrality measure, \emph{betweenness centrality},
within these two classes of trees. The betweenness of a node in a tree is the
number of paths passing through that node, and one can view it as an alternative
to the average distance metric that underlies the definition of a
centroid. In fact the average distance from a node to any other in a
graph also gives rise to a centrality measure, known simply as (inverse)
\emph{closeness centrality}. From this perspective, a centroid is simply a node
with maximal closeness. Both of these measures have proved useful in practical
applications~\citep{girvan02community,goh2002classification,shah11rumors}.

One finds, for example, that betweenness centrality must be rescaled linearly to
obtain limiting distributions over families of both simply generated and
increasing trees, implying that betweenness is typically~$\Theta(n)$ in both
classes. (A notable exception is that nodes close to the root of an increasing
tree tend to have betweenness that is $\Theta(n^2)$.) Moreover, the probability
that the centroid also has maximal betweenness approaches~$0.62$ and~$0.87$ in
labelled (simply generated) and recursive (increasing) trees
respectively~\citep{durant17centrality,durant17betweenness}.

\begin{figure}[p]
  \centering
  \includegraphics[width=.7\paperwidth]{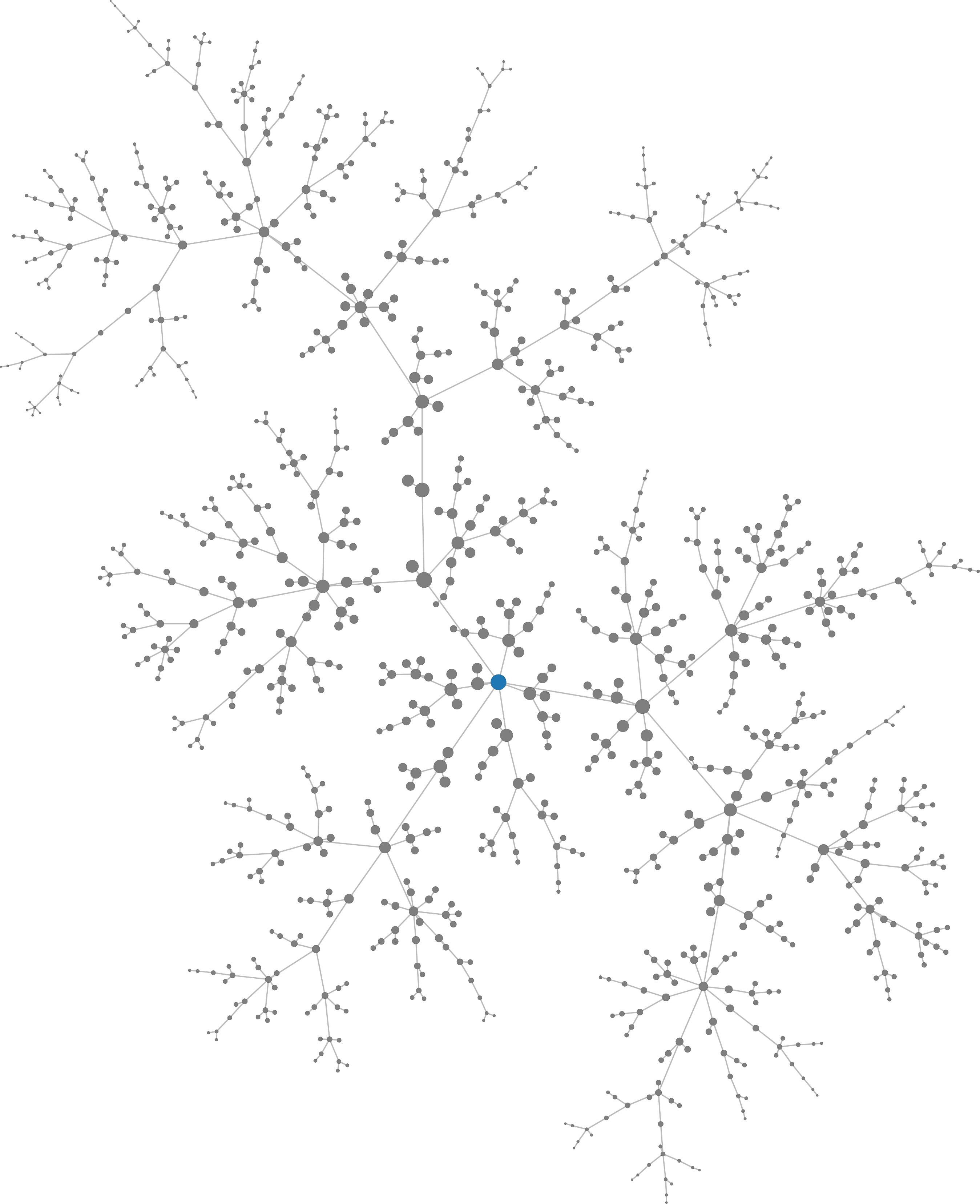}
\caption{A random recursive tree with~$n = 1000$ nodes. The larger a node,
the smaller its average distance to other nodes, and the centroid is depicted in
blue. In this particular example, the label of the centroid is~$2$, and the root
is the node of degree~$4$ directly to its right. (This figure can also be
compared to Figures~1 and~2 of~\citep{durant17centrality}.)}
\label{fig:recursive}
\end{figure}

\section{Very simple increasing trees}
\label{sec:trees}

As mentioned above, increasing trees are rooted, labelled trees that satisfy a
certain increasing property. One can state this property recursively: each of an
increasing tree's branches is itself an increasing tree whose labels are all
larger than that of the root.

In addition, specific families of increasing trees are defined concretely by
coupling a non-negative weight~$\phi_i$ to each node according to its
out-degree~$i$, and then letting the weight~$\omega(T)$ of the tree be the
product of the per-node weights. It is typical to assume $\phi_0 \ne 0$, and
$\phi_i > 0$ for some $i \ge 2$~\citep{bergeron92varieties}. Trees are understood
to be plane in this context, i.e., the order of the children of a node matters. However,
for some choices of weights (notably, the choice $\phi_i = \frac{1}{i!}$, which yields
recursive trees, see the discussion later), it is more natural to interpret the resulting
family of weighted trees as non-plane.

With these concepts in mind, one can construct a generating function for a
family~$\mc{T}$ of increasing trees. We let $y_n$ denote the total weight
of increasing trees in~$\mc{T}$ with $n$ nodes. The exponential
generating function will be denoted by $y(x)$:
\begin{equation*}
y(x) = \sum_{n \geq 1} \frac{y_n}{n!} x^n.
\end{equation*}
Recall that the symbolic act of
removing the node with the lowest label from every object in a class is
represented analytically by the differential operator $y(x) \rightarrow
y'(x)$~\citep[Section~VII.9.2]{flajolet09analytic}. We have:
\begin{equation}\label{eqn:y'(x) vs}
  y'(x) = \sum_{T\in\mc{T}} \omega(T) \frac{x^{\abs{T}-1}}{(\abs{T}-1)!}
      = \phi(y(x)),
\end{equation}
in which \begin{math}\phi(u) = \sum_i \phi_i u^i\end{math} is the family's characteristic weight
function. Note that increasing trees are labelled structures, so
all of the generating functions we make use of here will be \emph{exponential}
generating functions.

The so-called very simple increasing trees are a subclass of the increasing
trees defined in this way, and are important for two reasons: one, because the
most familiar examples of increasing trees fall into this category; and two,
because every family of very simple increasing trees can be characterised by a
few simple, useful properties.

\begin{lemma}[{\citet[Lemma~5]{panholzer07level}}]\label{lem:panholzer07}
Let $\mc{T}$ be a family of increasing trees; then $\mc{T}$ is a family of
\emph{very simple} increasing trees if the following (equivalent) properties
hold:
\begin{itemize}
  \item the total weight~$\abs{\mc{T}_n}$ of trees of size~$n$, also denoted
    by~$y_n$, satisfies \begin{math}y_{n+1}/y_n = c_1n + c_2\end{math} for certain $c_1, c_2 \in
    \reals$;
  \item repeatedly pruning the node with the largest label from a random tree 
in~$\mc{T}$ yields another, smaller, random tree; and
  \item trees can be constructed by way of a probabilistic growth process,
where the $n$-th node is attached to one of the previous $n-1$ nodes in the
$n$-th step, and the probability of being attached to a given node is proportional
to a linear function of its outdegree.
\end{itemize}
\end{lemma}

More directly, families of very simple increasing trees can be identified by
their characteristic functions, which always correspond to one of the following:
\begin{itemize}
  \item \emph{general recursive} trees:
    \begin{equation*} \phi(u) = \exp(c_1u), \text{ with } c_1 > 0; \end{equation*}
  \item \emph{general plane-oriented} trees:
    \begin{equation*} \phi(u) = (1 + c_2u)^{1+c_1/c_2}, \text{ with } c_2 < 0 \text{ and }
        c_1/c_2 < -1; \end{equation*}
  \item \emph{general $d$-ary increasing} trees:
    \begin{equation*} \phi(u) = (1 + c_2u)^{1+c_1/c_2}, \text{ with } c_2 > 0 \text{ and }
        c_1/c_2 \in \mb{Z}_{>0}. \end{equation*}
\end{itemize}
Note that in the above characteristic functions we have implicitly assumed that
the weight assigned to a leaf node is $1$, since the coefficient of $u^0$
in~$\phi(u)$---which we write as $[u^0]\phi(u)$---equals~$1$ in all three of
them. We have called these families `general' because each one has a certain
standard form, corresponding to certain fixed values of~$c_1$ and~$c_2$. These
are:
\begin{itemize}
  \item \emph{recursive} trees ($c_1 = 1$):
    \begin{equation*} \phi(u) = \exp(u) \implies y(x) = -\log(1-x); \end{equation*}
  \item \emph{plane-oriented} trees ($c_2 = -1$ and $c_1 = 2$):
    \begin{equation*} \phi(u) = (1-u)^{-1} \implies y(x) = 1 - \sqrt{1-2x}; \end{equation*}
  \item \emph{$d$-ary increasing} trees ($c_2 = 1$ and $c_1 = d-1 \in
    \mb{Z}_{>0}$):
    \begin{equation*} \phi(u) = (1+u)^d \implies y(x) = -1 + (1 - (d-1)x)^{-1/(d-1)}. \end{equation*}
\end{itemize}
Throughout the next few sections we will work with the more general forms, only
referring to specific cases for the sake of intuition or examples.

The probabilistic growth process mentioned in Lemma~\ref{lem:panholzer07} was
described briefly for the case of recursive trees in Section~\ref{sec:intro
results}, but is worth repeating in more detail here. For recursive trees, the
process starts with a root node---always labelled~$1$---and at each step,
node~$n$ is attached to one of the $n-1$ previous nodes, uniformly at random.
The tree obtained after the $n$th step is random (relative to its weight) in
$\mc{T}_n$. Clearly, the number of recursive trees of size~$n$ satisfies \begin{math}y_n =
(n-1) y_{n-1} = (n-1)!\end{math}.

The processes for plane-oriented and $d$-ary increasing trees are similar, but
with attachment probabilities that depend on the out-degrees of the existing
nodes. Firstly, in the plane-oriented case, a node with~$m$ children is viewed
as having $m+1$ distinct attachment points. Since there are a total of $2n-1$
such points in a tree of size~$n$, the number of plane-oriented trees is \begin{math}y_n =
(2n-3)y_{n-1} = (2n-3)!!\end{math}. Secondly, the defining characteristic of a $d$-ary
tree is that each of its nodes starts with $d$ attachment points, so that \begin{math}y_n =
((d-1)n + 1)y_{n-1}\end{math}. Note that these counts all abide by the general rule \begin{math}y_n
= (c_1n + c_2)y_{n-1}\end{math} of Lemma~\ref{lem:panholzer07}.

Generally, the probability of attaching node $n$ to one of the previous $n-1$ nodes
whose outdegree is $k$ is equal to \begin{math}(c_1 + c_2 - c_2k)/((n-1)c_1+c_2)\end{math}. Note
that this quotient only depends on $c_2/c_1$ rather than the precise values of $c_1$
and $c_2$, and that the proportionality factor $c_1+c_2 - c_2 k$ is independent of $n$.

For recursive trees, the probability simplifies to $1/(n-1)$ for all nodes; for plane-oriented
trees, we have \begin{math}c_1 + c_2 - c_2k = k+1\end{math}, so that the probability of attaching a new node to an existing
node increases with the number of children that node already has (so-called preferential
attachment). For $d$-ary trees, we have \begin{math}c_1 + c_2 - c_2k = d-k\end{math}, which is decreasing in $k$
and becomes $0$ for $k = d$, so that no more nodes can be attached once a node reaches full capacity.

\begin{remark}\label{rem:independence}
Since the probability that the new node $n$ is attached to a specific earlier node only depends on $n$ and the current
outdegree, but not the shape of the rest of the tree, it follows that the size and shape of the subtree rooted at $k$ are
independent of the history, i.e., the shape of the tree induced by the first $k$ nodes. This fact will be used a few times
in our arguments.
\end{remark}

\subsection{A common generating function} 

The standard expressions for the generating function~$y(x)$ that we have
mentioned above are all quite familiar and amenable to analysis, and this is
true of the generating functions for the more general families of very simple
increasing trees as well. However, for our purposes, it will be somewhat more
concise to work with the derived generating function~$y'(x)$, which has a
common, manageable form for all families:
\begin{equation}\label{eqn:y'(x)}
    y'(x) = (1 - c_1x)^{-(1+(c_2/c_1))}, \text{ with } c_1 \ne 0.
\end{equation}
This expression is not only a necessary property of very simple families, but a
sufficient one as well. To see this, note that it is usual to assume that $y_n
\ge 0$ for $n > 0$, along with $y_0 = 0$. Since \begin{math}y_n = \prod_{j=1}^{n-1} (c_1j +c_2)\end{math},
this is the case only if $c_1 > 0$ and $1 + c_2/c_1 > 0$, which
corresponds to general recursive and plane-oriented trees when $c_2 = 0$ and
$c_2 < 0$, respectively.

Another typical assumption is that \begin{math}\phi_i = [u^i]\phi(u) \ge 0\end{math} for $i > 0$,
with $\phi_0 > 0$ and $\phi_i > 0$ for some $i \ge 2$. This can also be applied
here, since by integrating $y'(x)$ when $c_2 > 0$ (the constant is determined
using $y_0 = 0$) it follows that \begin{math}y'(x) = \phi(y(x)) = (1 +
c_2y(x))^{1+(c_1/c_2)}\end{math}. This characteristic function satisfies the assumption
only if $c_1/c_2 \in \mathbb{Z}_{\ge0}$, although the case $c_1/c_2 = 0$ is
generally excluded to avoid the family of `path' trees.

Thus it is possible to derive results that are specific to very simple
increasing trees by working with the derived form~\eqref{eqn:y'(x)} and assuming
the non-negativity of the~$y_n$ and~$\phi_i$. That being said, our results will
remain unaffected even if we weaken the constraints on~$y_n$ and~$\phi_i$
slightly, so for the sections to come we define~$\alpha$ and assume that:
\begin{equation*} y'(x) = (1 - c_1x)^{-\alpha},
    \text{ with } \alpha = 1 + \frac{c_2}{c_1} > 0. \end{equation*}
In particular, we can write \begin{math}y_n = c_1^{n-1} \alpha\rf{n-1}\end{math}, in which
the latter term is a rising factorial. Recursive, plane-oriented, and $d$-ary
increasing trees now correspond to $\alpha = 1$, $\alpha = 1/2$, and $\alpha =
d/(d-1)$ respectively (for binary increasing trees, $\alpha = 2$).

\section{The depth of the centroid} 
\label{sec:depth}

The remaining sections deal with the depth, label, and subtree size of the
centroid, respectively. Intuitively, one can expect the root node to play a
large part in this analysis, because the branches of a random increasing tree
are typically very well balanced. To be more specific: the expected height of a
random very simple increasing tree
is~$\Theta(\log n)$~\cite[Chapter~6]{drmota09random}, and the depths of its
nodes follow a normal distribution with both mean and variance of
order~$\log n$~\citep{bergeron92varieties}. (Recall, e.g., that a strict binary
tree of size~$n$ has a height of \emph{at least} $\Theta(\log n)$.) This implies that
the average distance from the root to other nodes in the tree remains relatively
small, leading one to surmise that the centroid will lie at least somewhat close
to the root.

The starting point for all of our results is the observation that in a tree of
size~$n$, the node with label~$k$ is on the path between the root and the
centroid (always considering the nearer if there are two) if and only if $k$ has
at least $\fl{n/2}$ descendent nodes \citep{moon02centroid}. We let
$\Lambda_k(1/2)$ mark the occurrence of this event (later, we will also introduce a more
general version $\Lambda_k(\sigma)$). Our first goal will be to
derive a closed form for the probability~$\Pb_n(\Lambda_k(1/2))$, which, due to
its reliance on the label~$k$, will be obtained via the exponential generating
function~$y^{(k)}(x)$. By combining this closed form with a known probability
generating function for the depth of node $k$, we can describe the event that
$k$ is both at depth~$h$ and on the path from the root to the centroid, and by
marginalising over $k$, we arrive at a generating function that yields, in the
limit $n \to \infty$, the behaviour of the depth of the (nearest) centroid node.

\subsection{The probability of a node appearing on the centroid path} 
\label{sec:depth path}

There is, of course, a natural extension of $\Lambda_k(1/2)$ to the event
$\Lambda_k(\sigma)$ that node $k$ has at least $\fl{\sigma n}$ descendants,
where $1/2 \le \sigma < 1$, and in fact the closed form we desire for
$\Pb_n(\Lambda_k(1/2))$ is simply a special case of a similar expression for the
more general probability $\Pb_n(\Lambda_k(\sigma))$. Although we have no need
for it in determining the depth and label of the centroid, this more general
version will in fact be required in Section~\ref{sec:root branch}, when
considering the size of the centroid branch containing the root. Since the two
derivations are essentially identical, we deal with the variable case
immediately.

Consider the exponential generating function $y^{(k)}(x)$, which counts trees
(of size at least~$k$) as if the nodes $1$ through~$k$ were `size-less'. We
have:
\begin{align*}
    y^{(k)}(x) &= c_1^{k-1} \alpha\rf{k-1} (1 - c_1x)^{-(\alpha+k-1)} \\
      &= y_k \cdot y'(x) \cdot (1 - c_1x)^{-(k-1)},
\end{align*}
where the three terms can be interpreted as accounting for the configurations of
the first $k$ nodes, the subtree rooted at node~$k$, and the remaining subtrees,
respectively. A simple combinatorial interpretation for the latter can for example be given
for $d$-ary trees, where $c_1 = d-1$, $\alpha = d/(d-1)$ and \begin{math}y(x) = -1+(1-(d-1)x)^{-1/(d-1)}\end{math}.
When the structure of the first $k$ nodes is fixed, there are $(d-1)(k-1)$ places where a (possibly
empty) subtree can be attached to one of the first $k-1$ nodes. This is represented by the
exponential generating function \begin{math}(1+y(x))^{(d-1)(k-1)} = (1 - c_1x)^{-(k-1)}\end{math}.

With this in mind, the number (more accurately: the total weight)
of trees of size~$n$ in which $k$ has $m$ descendent nodes is:
\begin{equation*} y_k \binom{n-k}{m} \lr[\big]{m!\,[x^m]y'(x)}
    \lr*{(n-k-m)!\,[x^{n-k-m}](1 - c_1x)^{-(k-1)}}, \end{equation*}
and the proportion of trees in which $k$'s descendants number at least
$\fl{\sigma n}$ is, for $k > 1$:
\begin{align}\label{eqn:P_n(Lambda_k(sigma)) sum}
    \Pb_n(\Lambda_k(\sigma)) &= \frac{y_k(n-k)!}{y_n}
        \sum_{m=\fl{\sigma n}}^{n-k} \lr[\big]{[x^m]y'(x)}
        \lr*{[x^{n-k-m}](1 - c_1x)^{-(k-1)}} \nonumber \\
  &= \sum_{m=\fl{\sigma n}}^{n-k} \binom{\alpha+m-1}{m} \binom{n-m-2}{k-2}
      \Bigg/ \binom{\alpha+n-2}{n-k},
\end{align}
where we have used the fact that \begin{math}y_n = (n-k)!\,[x^{n-k}]y^{(k)}(x)\end{math}.

As long as the label~$k$ is small relative to the tree's size~$n$, we have an
asymptotic formula for $\Pb_n(\Lambda_k(\sigma))$ that is given in the following theorem.

\begin{theorem}\label{thm:P_n(Lambda_k(sigma))}
For $1 < k < \cl{(1-\sigma)n}$ such that $k = o(n^{1/4})$, the probability that
the node with label~$k$ has at least $\fl{\sigma n}$ descendants satisfies:
\begin{equation}\label{eqn:P_n(Lambda_k(sigma))}
    \Pb_n(\Lambda_k(\sigma)) = I_{1-\sigma}(k-1, \alpha)
        \lr*{1 + O\lr*{\frac{k^2}{\sqrt{n}}}},
\end{equation}
where the error term is uniform in~$\sigma$ over subsets of the form $[1/2,
1-\delta)$, for $0 < \delta < 1/2$, and
\begin{equation*} I_x(a, b) = \frac{\Beta(x; a, b)}{\Beta(a, b)} = \frac{\int_0^x t^{a-1}(1-t)^{b-1}\,dt}{\int_0^1 t^{a-1}(1-t)^{b-1}\,dt} \end{equation*}
is the regularised incomplete beta function.
\end{theorem}

\begin{proof}
The main step in going from~\eqref{eqn:P_n(Lambda_k(sigma)) sum}
to~\eqref{eqn:P_n(Lambda_k(sigma))} is an application of the Euler-Maclaurin
formula; but first, note that as $n$ grows:
\begin{align*}
    \Pb_n(\Lambda_k(\sigma)) &= \frac{\ds \sum_{m=\fl{\sigma n}}^{n-k}
        \frac{m^{\alpha-1}}{\Gamma(\alpha)} \lr*{1 + O\lr*{\frac{1}{m}}}
        \binom{n-m-2}{k-2}}
        {\ds \frac{(n-k)^{\alpha+k-2}}{\Gamma(\alpha+k-1)}
        \lr*{1 + \frac{(\alpha+k-1)\ff{2}}{2(n-k)}
          + O\lr*{\frac{k^4}{(n-k)^2}}}} \\
  &= \frac{n^{-(\alpha+k-2)}}{\Beta(k-1, \alpha)}
        \sum_{m=k}^{\cl{(1-\sigma)n)}} (n-m)^{\alpha-1} (m - 2)\ff{k-2}
        \lr*{1 + O\lr*{\frac{k^2}{n}}}.
\end{align*}
Splitting the sum at an intermediate value~$n^{1-\varepsilon}$, where $0 <
\varepsilon < 1$, reveals that the contribution from smaller values of~$m$ is
minimal:
\begin{equation}\label{eqn:P_n(Lambda_k(sigma)) small}
    n^{-(\alpha+k-2)} \sum_{m=k}^{\cl{n^{1-\varepsilon}}-1} (n-m)^{\alpha-1}
        (m-2)\ff{k-2} = O\lr*{n^{-(k-1)\varepsilon}}.
\end{equation}
It suffices now to apply the Euler-Maclaurin formula to the dominant portion of
the sum:
\begin{align}\label{eqn:P_n(Lambda_k(sigma)) large}
    &n^{-(\alpha+k-2)} \sum_{m=\cl{n^{1-\varepsilon}}}^{\cl{(1-\sigma)n}}
        (n-m)^{\alpha-1} m^{k-2} \lr*{1 + O\lr*{\frac{k^2}{m}}} \nonumber \\
  =\; &n^{-(\alpha+k-2)} \int_{\cl{n^{1-\varepsilon}}}^{\cl{(1-\sigma)n}}
        (n-u)^{\alpha-1} u^{k-2}\, du \lr*{1
      + O\lr*{\frac{k^2}{n^{1-\varepsilon}}}} + O\lr*{\frac1n} \nonumber \\
  =\; &\int_{n^{-\varepsilon}+O(1/n)}^{1-\sigma+O(1/n)} t^{k-2}
        (1-t)^{\alpha-1}\, dt \lr*{1 + O\lr*{\frac{k^2}{n^{1-\varepsilon}}}}.
\end{align}
Note that in absorbing the error term~$O(1/n)$ we have implicitly treated
$\sigma$ as a constant with respect to~$n$, and as such, the uniformity of the
error term only holds as long as $\sigma$ is not allowed to tend arbitrarily
close to~$1$. Also, as long as $\varepsilon \ge 1/k$, the term
$k^2/n^{1-\varepsilon}$ is not smaller in order than that of
equation~\eqref{eqn:P_n(Lambda_k(sigma)) small}, and the contribution of the
first portion of the sum can be ignored. This is the case for all $k > 1$ when
$1/2 \le \varepsilon < 1$.

As $n$ grows, the bounds of the integral in~\eqref{eqn:P_n(Lambda_k(sigma))
large} approach $0$ and~$1-\sigma$ at the following rates:
\begin{equation*} \int_0^{n^{-\varepsilon}+O(1/n)} t^{k-2} (1-t)^{\alpha-1}\, dt
  = O\lr*{n^{-(k-1)\varepsilon}}, \end{equation*}
and
\begin{equation*} \int_{1-\sigma}^{1-\sigma+O(1/n)} t^{k-2} (1-t)^{\alpha-1}\, dt
  = O\lr*{\frac1n}. \end{equation*}
Noting that these terms are also of orders lower than $k^2/n^{1-\varepsilon}$
when $1/2 \le \varepsilon < 1$, we see that the probability of node~$k$ having
at least $\fl{\sigma n}$ descendants can be written, for $1 < k =
o(n^{(1-\varepsilon)/2})$, as:
\begin{align*}
    \Pb_n(\Lambda_k(\sigma)) &= \frac{1}{\Beta(k-1, \alpha)}
        \int_0^{1-\sigma} t^{k-2} (1-t)^{\alpha-1}\, dt
        \lr*{1 + O\lr*{\frac{k^2}{n^{1-\varepsilon}}}} \\
  &= I_{1-\sigma}(k-1, \alpha)
  \lr*{1 + O\lr*{\frac{k^2}{n^{1-\varepsilon}}}}. \qedhere
\end{align*}
\end{proof}

We mention also that an alternative representation of $\Pb_n(\Lambda_k(\sigma))$
can be obtained using the binomial theorem, since
\begin{align*}
    \int_0^{1-\sigma} t^{k-2} (1-t)^{\alpha-1}\, dt &= \Beta(k-1, \alpha)
      - \int_{1-\sigma}^1 t^{k-2} (1-t)^{\alpha-1}\, dt \\
  &= \Beta(k-1, \alpha) - \sum_{l=0}^{k-2} \binom{k-2}{l}
        \frac{(-1)^l}{l+\alpha}\sigma^{l+\alpha}.
\end{align*}

\begin{corollary}\label{cor:P_n(Lambda_k(1/2))}
For $1 < k = o(n^{1/4})$, the probability that node~$k$ is on the path from the
root to the (nearest) centroid node satisfies:
\begin{equation*} \Pb_n(\Lambda_k(1/2)) = I_{1/2}(k-1, \alpha)
    \lr*{1 + O\lr*{\frac{k^2}{\sqrt{n}}}}. \end{equation*}
\end{corollary}

Finally, we give limiting probabilities for the event that node~$k$ is on the
path to the centroid in the two simplest families of very simple increasing
trees, for which the incomplete beta function can be easily simplified:
\begin{corollary}
For recursive trees:
\begin{equation*} \lim_{n\to\infty} \Pb_n(\Lambda_k(1/2)) = I_{1/2}(k-1, 1) = 2^{-(k-1)}. \end{equation*}
For binary increasing trees:
\begin{equation*} \lim_{n\to\infty} \Pb_n(\Lambda_k(1/2))
  = I_{1/2}(k-1, 2) = (k+1) 2^{-k}. \end{equation*}
\end{corollary}

\subsection{A uniform bound on the path probability} 

In addition to the asymptotic expression for $\Pb_n(\Lambda_k(\sigma))$ given
above, we can show that the probability of a specific node~$k$ appearing on the
path to the centroid not only vanishes for large~$k$, but does so exponentially
in~$k$ and uniformly over~$n$. In particular:
\begin{equation}\label{eqn:P_n(Lambda_k(1/2)) bound}
    \Pb_n(\Lambda_k(1/2)) \le \frac{\alpha\rf{k-1}}{(k-1)!} 2^{-(k-2)}.
\end{equation}
It is this fact, in combination with Corollary~\ref{cor:P_n(Lambda_k(1/2))},
that will allow us to derive limiting distributions for events that depend on
$\Pb_n(\Lambda_k(1/2))$.

Once again a more general form of this result will be required later, in
Section~\ref{sec:root branch}, so to avoid a repeated derivation we give the
version for variable~$\sigma$ here.

\begin{lemma}\label{lem:P_n(Lambda_k(sigma)) bound}
For $k \ge 1$ and $1/2 \le \sigma < 1$, the probability that node~$k$ has at
least $\fl{\sigma n}$ descendants in a tree of size $n \ge 3$ is subject to an
upper bound that decreases exponentially with~$k$:
\begin{equation}\label{eqn:P_n(Lambda_k(sigma)) bound}
    \Pb_n(\Lambda_k(\sigma)) \le \frac{3}{\sigma}
        \frac{\alpha\rf{k-1}}{(k-1)!} (1-\sigma)^{k-1}.
\end{equation}
\end{lemma}

\begin{proof}
Firstly, take note of the following inequality involving binomial coefficients:
if $\alpha \in \reals_{\ge0}$ and $m \le n \in \mb{Z}_{\ge0}$, then:
\begin{align}\label{eqn:binomial identity}
    \binom{m-1+\alpha}{m} &= \frac{(\alpha+m-1)\ff{m}}{m!} \nonumber \\
  &\le \frac{(\alpha+n) \dotsm (\alpha+m)}{n \dotsm m}
        \frac{(\alpha+m-1)\ff{m}}{m!} \nonumber \\
  &= \frac{n+\alpha}{m} \binom{n-1+\alpha}{n}.
\end{align}
Since $\Pb_n(\Lambda_1(\sigma)) = 1$, and $\Pb_n(\Lambda_k(\sigma)) = 0$
whenever $k > \cl{(1-\sigma)n}$, we need only consider $1 < k \le
\cl{(1-\sigma)n}$. In this case:
\begin{align*}
    \Pb_n(\Lambda_k(\sigma)) &= \sum_{m=\fl{\sigma n}}^{n-k}
        \binom{m-1+\alpha}{m} \binom{n-m-2}{k-2} \Bigg/
        \binom{n-2+\alpha}{n-k} \\
&\leq \sum_{m=\fl{\sigma n}}^{n-k} \frac{n-k+\alpha}{m} 
        \binom{n-k-1+\alpha}{n-k} \binom{n-m-2}{k-2} \Bigg/
        \binom{n-2+\alpha}{n-k} \\
&\leq \sum_{m=\fl{\sigma n}}^{n-k} \frac{n-k+\alpha}{\fl{\sigma n}} 
        \binom{n-k-1+\alpha}{n-k} \binom{n-m-2}{k-2} \Bigg/
        \binom{n-2+\alpha}{n-k} \\
 &= \frac{n-k+\alpha}{\fl{\sigma n}} \binom{n-k-1+\alpha}{n-k}
        \binom{\cl{(1-\sigma)n} - 1}{k-1} \Bigg/ \binom{n-2+\alpha}{n-k} \\
  &= \frac{n-k+\alpha}{\fl{\sigma n}}
        \frac{(\cl{(1-\sigma)n} - 1)\ff{k-1}}{(k-1)!}
        \frac{\Gamma(n-k+\alpha)}{\Gamma(n-1+\alpha)}
        \frac{\Gamma(k-1+\alpha)}{\Gamma(\alpha)} \\
  &= \frac{\alpha\rf{k-1}}{(k-1)!} \frac{n-k+\alpha}{\fl{\sigma n}}
        \frac{(\cl{(1-\sigma)n} - 1)\ff{k-1}}{(n-2+\alpha)\ff{k-1}} \\
  &\le \frac{\alpha\rf{k-1}}{(k-1)!} \frac{n-k+\alpha}{n-2}
        \frac{n(n-2) \dotsm (n-2k+4)}{(n-2+\alpha) \dotsm (n-k+\alpha)}
        \frac{(1-\sigma)^{k-1}}{\sigma} \\
  &\le \frac{\alpha\rf{k-1}}{(k-1)!} \frac{n}{n-2}
        \frac{(1-\sigma)^{k-1}}{\sigma}.
\end{align*}
This yields the stated bound whenever $n \ge 3$. In the specific case $\sigma =
1/2$, the final two lines can be improved slightly, resulting in
equation~\eqref{eqn:P_n(Lambda_k(1/2)) bound}:
\begin{align*}
\frac{\alpha\rf{k-1}}{(k-1)!} \frac{n-k+\alpha}{\fl{n/2}}
        \frac{(\cl{n/2} - 1)\ff{k-1}}{(n-2+\alpha)\ff{k-1}} &\leq \frac{\alpha\rf{k-1}}{(k-1)!}\frac{(\cl{n/2} - 2)\ff{k-2}}{(n-2+\alpha)\ff{k-2}} \\
&= \frac{\alpha\rf{k-1}}{(k-1)!} \prod_{j=1}^{k-2} \frac{\cl{n/2} - 1 - j}{n-1+\alpha-j} \\
&\leq \frac{\alpha\rf{k-1}}{(k-1)!} \prod_{j=1}^{k-2} \frac{(n+1)/2 - 1 - j}{n-1-j} \\
&= \frac{\alpha\rf{k-1}}{(k-1)!} 2^{-(k-2)} \prod_{j=1}^{k-2} \frac{n -1-2j}{n-1-j} \\
&\leq \frac{\alpha\rf{k-1}}{(k-1)!} 2^{-(k-2)}. \qedhere
\end{align*}
\end{proof}

\subsection{A limiting distribution for the depth of the centroid} 

Let $D = D(T)$ denote the depth of the centroid---that is, the number of edges on
the path from the root to the centroid node (the nearer if there are two)---in a
random tree~$T$. As mentioned earlier, $D(T) \ge h$ if and only if there is a
vertex at depth~$h$ that is on this path. Breaking this event down per label, we
may say that the depth of the centroid is at least $h$ if and only if for some
label~$k$, node~$k$ has both depth~$h$ and is present on the path to the
centroid. Since these per-label events are disjoint, and the depth of node $k$ is
independent of the size of the subtree rooted at $k$ (see Remark~\ref{rem:independence}),
we have
\begin{equation*} \Pb_n(D \ge h) = \sum_{k\ge1} \Pb_n(D_k = h) \Pb_n(\Lambda_k(1/2)), \end{equation*}
in which $D_k$ is a random variable over the possible depths $h \in
\mb{Z}_{\ge0}$ of node~$k$. This random variable has a known probability
generating function \citep{panholzer07level}:
\begin{equation}\label{eqn:D_k pgf}
    \sum_{h\ge0} \Pb_n(D_k = h) v^h = \prod_{j=0}^{k-2}
        \frac{\alpha v + j}{\alpha+j}
      = \frac{(\alpha v)\rf{k-1}}{\alpha\rf{k-1}},
\end{equation}
which is independent of $n$, implying that \begin{math}\Pb_n(D_k = h) = \Pb(D_k = h)\end{math} is as
well. Combining these two expressions yields a (complementary cumulative)
probability generating function for the depth of the centroid:
\begin{align}\label{eqn:C_n(v)}
    C_n(v) = \sum_{h\ge0} \Pb_n(D \ge h) v^h &= \sum_{k\ge1}
        \Pb_n(\Lambda_k(1/2)) \sum_{h\ge0} \Pb(D_k = h) v^h \nonumber \\
  &= \sum_{k\ge1} \Pb_n(\Lambda_k(1/2))
    \frac{(\alpha v)\rf{k-1}}{\alpha\rf{k-1}}.
\end{align}
Our goal is to make use of the asymptotic form of $\Pb_n(\Lambda_k(1/2))$ given
in Corollary~\ref{cor:P_n(Lambda_k(1/2))} to find the limit of this generating
function, and then simply extract the desired probabilities as coefficients.

\begin{theorem}\label{thm:P(D = h)}
The depth $D(\mc{T}_n)$ of the centroid node in a random tree of size~$n$
converges in probability to the discrete random variable $\mc{D}$ supported by
$\mb{Z}_{\ge0}$ and with cumulative distribution function:
\begin{equation*} \Pb(\mc{D} \ge h) = \lr*{\frac{\alpha}{\alpha-1}}^h
    \lr*{1 - 2^{1-\alpha} \sum_{j=0}^{h-1} \frac{((\alpha-1)\ln2)^j}{j!}} \end{equation*}
and mass function:
\begin{equation*} \Pb(\mc{D} = h) = \frac{\alpha^h}{(\alpha-1)^{h+1}}
    \lrb*{2^{1-\alpha} \lr*{\sum_{j=0}^{h-1}
    \frac{\lr[\big]{(\alpha-1)\ln2}^j}{j!}
      + \frac{\alpha\lr[\big]{(\alpha-1)\ln2}^h}{h!}}\! - 1}. \end{equation*}
\end{theorem}

In the case $\alpha = 1$ (corresponding to recursive trees), the expressions are undefined. 
Singular cases such as this---which will appear throughout this paper---can be derived by direct
calculations or as limits of the more general cases.

\begin{corollary}[\citet{moon02centroid}]
For recursive trees:
  \begin{equation*} \lim_{n\to\infty} \Pb_n(D \ge h) = \frac{(\ln2)^h}{h!} \end{equation*}
and
  \begin{equation*} \lim_{n\to\infty} \Pb_n(D = h) = \frac{(\ln2)^h}{h!} - \frac{(\ln2)^{h+1}}{(h+1)!}. \end{equation*}
\end{corollary}

\begin{proof}[of Theorem~\ref{thm:P(D = h)}]
Letting $n$ be large, and fixing $K = \fl{n^{1/4-\varepsilon}}$ for an
arbitrarily small $\varepsilon > 0$, Corollary~\ref{cor:P_n(Lambda_k(1/2))} and
equation~\eqref{eqn:C_n(v)} above imply:
\begin{align*}
  C_n(v) = 1 &+ \sum_{k=2}^K I_{1/2}(k-1, \alpha)
        \lr*{1 + O\lr*{\frac{k^2}{\sqrt{n}}}}
        \frac{(\alpha v)\rf{k-1}}{\alpha\rf{k-1}} \\
    &+ \sum_{k=K+1}^{\cl{n/2}} \Pb_n(\Lambda_k(1/2))
        \frac{(\alpha v)\rf{k-1}}{\alpha\rf{k-1}}.
\end{align*}
As $n \to \infty$, assuming $\abs{v} < 1$, the second sum tends to $0$: recalling the inequalities~\eqref{eqn:P_n(Lambda_k(1/2)) bound} and~\eqref{eqn:binomial identity}, we get
\begin{align*}
    \sum_{k=K+1}^{\cl{n/2}} \abs*{\Pb_n(\Lambda_k(1/2))
      \frac{(\alpha v)\rf{k-1}}{\alpha\rf{k-1}}}
  &\le \sum_{k=K+1}^{\cl{n/2}} \frac{(\alpha v)\rf{k-1}}{(k-1)!} 2^{-(k-2)} \\
  &\le \sum_{k=K+1}^{\cl{n/2}} \frac{\alpha\rf{k-1}}{(k-1)!} 2^{-(k-2)} \\
  &= 2\sum_{k=K}^{\cl{n/2}-1} \binom{k-1+\alpha}{k} 2^{-k} \\
  &\le n \frac{\frac{n}{2} + \alpha}{\fl{n^{1/4-\varepsilon}}}
        \binom{\frac{n}{2} - 1 + \alpha}{\frac{n}{2}}
        2^{-\fl{n^{1/4-\varepsilon}}} \\
  &= O\lr*{\frac{n^{\alpha+3/4+\varepsilon}}{2^{\fl{n^{1/4-\varepsilon}}}}}
      \tendsto{n\to\infty} 0,
\end{align*}
where we have made use of the upper bound on $\Pb_n(\Lambda_k(1/2))$ given
in equation~\eqref{eqn:P_n(Lambda_k(1/2)) bound}, as well as the binomial
inequality~\eqref{eqn:binomial identity}. Similarly, the tail which replaces
this sum is negligible:
\begin{align*}
    \sum_{k>K} \abs*{I_{1/2}(k-1, \alpha)
        \frac{(\alpha v)\rf{k-1}}{\alpha\rf{k-1}}}
  &\le \sum_{k>K} \Beta(k-1, \alpha)^{-1}
        \int_0^{1/2} t^{k-2} (1-t)^{\alpha-1}\, dt \\
  &\le \alpha \sum_{k>K} \binom{\alpha+k-2}{k-2} 2^{-(k-1)} \\
  &= O\lr*{K^\alpha 2^{-K}} \tendsto{n\to\infty} 0.
\end{align*}
Thus, for values of $v$ within the unit circle, the pointwise limit of $C_n(v)$
is:
\begin{align}\label{eqn:C(v)}
    C(v) = \lim_{n\to\infty} C_n(v) &= 1 + \sum_{k\ge2} I_{1/2}(k-1, \alpha)
        \frac{\Gamma(\alpha v + k - 1) \Gamma(\alpha)}
        {\Gamma(\alpha v) \Gamma(\alpha+k-1)} \nonumber \\
  &= 1 + \sum_{k\ge2} \frac{\Gamma(\alpha v + k - 1)}
        {\Gamma(\alpha v) \Gamma(k-1)}
        \int_0^{1/2} t^{k-2} (1-t)^{\alpha-1}\, dt \nonumber \\
  &= 1 + \alpha v \int_0^{1/2} (1-t)^{\alpha-1} \sum_{k\ge2}
        \binom{\alpha v + k - 2}{k-2} t^{k-2}\, dt \nonumber \\
  &= 1 + \alpha v \int_0^{1/2} (1-t)^{-\alpha(v-1)-2}\, dt \nonumber \\
  &= 1 + \frac{\alpha v}{\alpha(v-1) + 1} \lr*{2^{\alpha(v-1)+1} - 1}.
\end{align}
Considering now the probability generating function involving the mass
$P_n(D = h)$---say $A_n(v)$---we see that this too converges:
\begin{align}\label{eqn:A(v)}
    A_n(v) = \sum_{h\ge0} \Pb_n(D = h) v^h
  &= C_n(v) - \frac{C_n(v)-1}{v} \nonumber \\
  &\to C(v) - \frac{C(v)-1}{v} \nonumber \\
  &= 1 + \frac{\alpha(v-1)}{\alpha(v-1) + 1} \lr*{2^{\alpha(v-1)+1} - 1} = A(v).
\end{align}
Since pointwise convergence of probability generating functions (in this case
$A_n(v) \to A(v)$) implies convergence in probability of their distributions
\citep[Theorem~IX.1]{flajolet09analytic}, we have the stated theorem via
\begin{math}\Pb_n(D \ge h) \to [v^h]C(v)\end{math}. The limiting mass function follows naturally.

Let us finally remark that for the special case $\alpha=1$, equation~\eqref{eqn:C(v)}
directly yields
\begin{equation*} \lim_{n\to\infty} \Pb_n(D \ge h) = [v^h] 2^v
  = \frac{(\ln2)^h}{h!}. \qedhere \end{equation*}
\end{proof}

\subsection{Moments of the depth distribution} 

To close our discussion of the centroid's depth, we consider the moments of
$D(T)$ as the size of the tree tends to infinity. More specifically,
with~$C_n(v)$ and~$A_n(v)$ as they were in the proof of
Theorem~\ref{thm:P(D = h)} (along with their respective limits), we are
interested in:
\begin{equation*} \lim_{n\to\infty} \Ex_n(D\ff{m})
  = \lim_{n\to\infty} \sum_{h\ge m} h\ff{m} \Pb_n(D = h)
  = \lim_{n\to\infty} A_n^{(m)}(1). \end{equation*}
We show firstly that the moments of $\Pb_n(D = h)$ converge to those of its
limiting distribution $\Pb(D = h)$, and then, instead of dealing with $A_n(v)$
directly, derive the moments' asymptotic behaviour using the limiting generating
function $A(v)$.

\begin{lemma}\label{lem:E(D) convergence}
The moments of the distribution of the centroid's depth $D(\mc{T}_n)$
converge to those of $\mc{D}$, i.e.:
\begin{equation*} \lim_{n\to\infty} \Ex_n(D^m) = \Ex(\mc{D}^m). \end{equation*}
\end{lemma}

\begin{proof}
This follows from Lemma~\ref{lem:P_n(Lambda_k(sigma)) bound} and Lebesgue's
dominated convergence theorem, which states that if $(f_n)$ is a sequence of
real-valued functions, and $g$ a function such that $\abs*{f_n} \le g$ for all
$n$, then if $\int g < \infty$, one has \begin{math}\lim_{n\to\infty} \int f_n = \int
\lim_{n\to\infty} f_n\end{math}.

For our purposes, let \begin{math}f_n(h) = h\ff{m} \Pb_n(D = h)\end{math}. The bound on
$\Pb_n(\Lambda_k(1/2))$ given in equation~\eqref{eqn:P_n(Lambda_k(1/2)) bound}
then leads to a similar one (also uniform over $n$) on $\Pb_n(D = h)$, as
follows (recall that \begin{math}[v^h]C_n(v) = \Pb_n(D \ge h)\end{math}):
\begin{align*}
    \Pb_n(D = h) \le [v^h]C_n(v) &= [v^h]\sum_{k\ge1} \Pb_n(\Lambda_k(1/2))
        \frac{(\alpha v)\rf{k-1}}{\alpha\rf{k-1}} \\
  &\le [v^h]\sum_{k\ge1} \binom{\alpha v + k - 2}{k-1} 2^{-(k-2)} \\
  &= [v^h]2^{1+\alpha v} = 2\frac{(\alpha\ln2)^h}{h!}.
\end{align*}
Since the range of the random variable~$D$ is countable, the integrals in
Lebesgue's dominated convergence theorem become sums, and we are left with
\begin{equation*} 2\sum_{h\ge0} h\ff{m}\, \frac{(\alpha\ln2)^h}{h!}
  = (\alpha\ln2)^m\, 2^{\alpha+1} < \infty. \end{equation*}
Thus the factorial moments of the distributions $\Pb_n(D = h)$ converge to
those of $\Pb(\mc{D} = h)$, and since the usual higher-order moments
$\Ex(\mc{D}^m)$ are (finite) linear combinations of the factorial moments,
convergence holds for them as well.
\end{proof}

With convergence established, all that remains is to compute the moments of
$\mc{D}$ by making use of its probability generating function \begin{math}A(v) =
\sum_{h\ge0} \Pb(\mc{D} = h) v^h\end{math}, since \begin{math}\Ex(\mc{D}\ff{m}) =
A^{(m)}(1)\end{math}.

\begin{theorem}\label{thm:E(D_star)}
The limit of the $m$th factorial moment of the centroid's depth $D(\mc{T}_n)$ is
given by:
\begin{equation*} \Ex(\mc{D}\ff{m}) = m\alpha^{m} \lr*{2\sum_{j=0}^{m-2}
    \binom{m-1}{j} (-1)^j j!\, (\ln2)^{m-1-j} + (-1)^{m-1}(m-1)!}. \end{equation*}
In particular, the limits of its mean and variance are:
\begin{gather*}
    \Ex(\mc{D}) = \alpha, \\
    \Va(\mc{D}) = \alpha^2(4\ln2 - 3) + \alpha.
\end{gather*}
\end{theorem}

\begin{proof}
The calculation can be simplified slightly by writing the derivative of the
expression in equation~\eqref{eqn:A(v)} as:
\begin{align*}
    A^{(m)}(v) &= \frac{d^m}{dv^m} \lrb*{\alpha(v-1) \cdot
        \lr[\big]{1 + \alpha(v-1)}^{-1} \cdot \lr*{2^{\alpha(v-1)+1} - 1}} \\
  &= \frac{d^m}{dv^m} \lrb*{a(v) \cdot b(v) \cdot c(v)} \\
  &= \sum_{i+j+k=m} \binom{m}{i, j, k} a^{(i)}(v)\, b^{(j)}(v)\, c^{(k)}(v).
\end{align*}
Since we are interested in $v = 1$, note that $a'(1) = \alpha$, but
$a^{(i)}(1) = 0$ when $i \ne 1$. Furthermore, $b^{(j)}(1) = (-\alpha)^j j!$ for
$j \ge 0$, and $c^{(k)}(1) = 2(\alpha\ln2)^k$ for $k > 0$, whereas $c(1) = 1$.
This leads to:
\begin{align*}
    \Ex(\mc{D}\ff{m}) &= m\alpha\sum_{j+k=m-1} \binom{m-1}{j, k}
        b^{(j)}(1)\, c^{(k)}(1) \\
  &= m\alpha^{m} \lr*{2\sum_{j=0}^{m-2} \binom{m-1}{j} (-1)^j j!\,
        (\ln2)^{m-1-j} + (-1)^{m-1}(m-1)!}.
\end{align*}
The mean is computed more simply as $C(1) - 1 = \alpha$, and the second
factorial moment is \begin{math}\Ex(\mc{D}\ff{2}) = 2\alpha^2(2\ln2 - 1)\end{math}.
\end{proof}

Finally, we make two small remarks again: that the limits of the mean and
variance of the depth of the centroid in a random recursive tree ($1$ and $4\ln2
- 2$ respectively) were previously given by \citet{moon02centroid}; and that
Theorem~\ref{thm:E(D_star)} implies that the mean and variance of the
centroid's depth are greatest (in the limit) in the case of binary increasing
trees ($\alpha = 2$).

\section{The label of the centroid}
\label{sec:label}

Our second task regarding the centroid of an increasing tree is to describe its
label, which we will denote by~$L = L(T)$.

Since we will have no need for the general event $\Lambda_k(\sigma)$ throughout
this section, let us adopt the shorthand $\Lambda_k = \Lambda_k(1/2)$ to denote
the presence of label~$k$ on the path between the root and centroid nodes. A key
observation is that the event $L(T) = k$ can be expressed in terms of the
presence (or lack thereof) of nodes $k, k+1, \dotsc$ on this path, namely:
\begin{equation}\label{eqn:P_n(L = k)}
\begin{aligned}
    \Pb_n(L = k) = \Pb_n(\Lambda_k)
      &- \Pb_n(\Lambda_k \cap \Lambda_{k+1}) \\
      &- \Pb_n(\Lambda_k \cap \ol{\Lambda}_{k+1} \cap \Lambda_{k+2}) \\
      &- \Pb_n(\Lambda_k \cap \ol{\Lambda}_{k+1} \cap \ol{\Lambda}_{k+2} \cap \Lambda_{k+3}) \\
      &- \dotsb.
\end{aligned}
\end{equation}
Here $\ol{\Lambda}_l$ is the complement of $\Lambda_l$, i.e., it is the event
that node~$l$ is \emph{not} on the path to the centroid.
Equation~\eqref{eqn:P_n(L = k)} simply states that the centroid has label~$k$ if
and only if~$k$ is on the path to the centroid, but none of the nodes $k+1, k+2,
\dotsc$ are. One can write a similar expression for the probability that the
centroid's label is at least $k$:
\begin{equation}\label{eqn:P_n(L >= k)}
\begin{aligned}
    \Pb_n(L \ge k) = \Pb_n(\Lambda_k)
      &+ \Pb_n(\ol{\Lambda}_k \cap \Lambda_{k+1}) \\
      &+ \Pb_n(\ol{\Lambda}_k \cap \ol{\Lambda}_{k+1} \cap \Lambda_{k+2}) \\
      &+ \Pb_n(\ol{\Lambda}_k \cap \ol{\Lambda}_{k+1} \cap \ol{\Lambda}_{k+2} \cap \Lambda_{k+3}) \\
      &+ \dotsb.
\end{aligned}
\end{equation}
The composite event \begin{math}\Lambda_k \cap \ol{\Lambda}_{k+1} \cap \dotsb \cap
\ol{\Lambda}_{k+j-1} \cap \Lambda_{k+j}\end{math} in equation~\eqref{eqn:P_n(L = k)}
requires that $k$ and $k+j$ appear on the path to the centroid, but none of
$k+1, \dotsc, k+j-1$ do. This occurs if and only if $k+j$ is on the path and has
node~$k$ as its parent. (This is a simpler condition than the one required by
equation~\eqref{eqn:P_n(L >= k)}, which would be that node $k+j$ is on the path
and its parent is any one of the nodes $1, \dotsc, k-1$.)

Let $A_l(T)$ be the random variable that yields node~$l$'s parent, which, if we
consider the increasing tree's probabilistic growth process, is the node~$l$ was
`attached' to. Then we are interested in \begin{math}\Pb_n(\Lambda_{k+j} \cap
(A_{k+j} = k))\end{math}, for fixed $k$ and~$j$, as $n \to \infty$. Because the size of
the subtree consisting of node $k+j$ and its descendants is independent of the
node to which $k+j$ was attached (see Remark~\ref{rem:independence}), we have:
\begin{equation}\label{eqn:P_n(Lambda cap A)}
    \Pb_n\lr[\big]{\Lambda_{k+j} \cap (A_{k+j} = k)} = \Pb_n(\Lambda_{k+j})
        \Pb(A_{k+j} = k),
\end{equation}
where the second probability in the product is independent of the size~$n$ of
the tree. Since by Corollary~\ref{cor:P_n(Lambda_k(1/2))} we already know the
asymptotic behaviour of $\Pb_n(\Lambda_{k+j})$, 
we need an expression for the probability that node $k+j$ attaches to node~$k$.
Such an expression can be found in the literature:

\begin{lemma}[see \citet{dobrow96poisson,kuba10depths}]\label{lem:P(A = k)}
For $k, j \in \mb{Z}_{>0}$, the probability that the parent of node $k+j$ has
label~$k$ is given by:
\begin{equation*} \Pb(A_{k+j} = k) = \frac{\alpha k\rf{j-1}}{(\alpha+k-1)\rf{j}}, \end{equation*}
which does not depend on the size of the tree.
\end{lemma}
Note that in the case of recursive trees ($\alpha = 1$), the
probability of a particular attachment is \begin{math}\Pb(A_{k+j} = k) = 1/(k+j-1)\end{math}, which
agrees with the family's growth process.

\subsection{A limiting distribution for the label of the centroid} 

Following on from equations~\eqref{eqn:P_n(L = k)} and~\eqref{eqn:P_n(Lambda
cap A)}, we can write the probability of the centroid assuming a certain
label~$k$ in terms of the events $\Lambda_k$ and $A_{k+j} = k$ (for which we
have closed-form expressions):
\begin{equation*} \Pb_n(L = k) = \Pb_n(\Lambda_k) - \sum_{j\ge1} \Pb_n(\Lambda_{k+j})
    \Pb_n(A_{k+j} = k). \end{equation*}
As a consequence, we arrive at the desired result of this section:

\begin{theorem}\label{thm:P(L = k)}
The label $L(\mc{T}_n)$ of the centroid node in a random tree of size~$n$
converges in probability to a discrete random variable $\mc{L}$ supported by
$\mb{Z}_{\ge0}$ and with mass function:
\begin{equation*} \Pb(\mc{L} = k) = \begin{cases}
    1 - \frac{\alpha}{\alpha-1} \lr*{1 - 2^{-(\alpha-1)}}
        &\text{ if } k = 1, \\
    I_{1/2}(k-1, \alpha) - \frac{\alpha}{\alpha-1} I_{1/2}(k, \alpha-1)
        &\text{ otherwise.} \end{cases} \end{equation*}
An alternative form, which holds for $k \ge 1$, is given by:
\begin{equation}\label{eqn:P(L = k)}
    \Pb(\mc{L} = k) = -\frac{1}{\alpha-1} I_{1/2}(k, \alpha)
  + \lr*{1 + \frac{\alpha}{\alpha-1}} \binom{\alpha+k-2}{k-1} 2^{-(\alpha+k-1)}.
\end{equation}
\end{theorem}

\begin{proof}
Recalling the asymptotic expression for $\Pb_n(\Lambda_k)$
(Corollary~\ref{cor:P_n(Lambda_k(1/2))}), assume now that $n$ is large, and fix
$J$ so that $k+J = \fl{n^{1/4-\varepsilon}}$, for some arbitrarily small,
positive $\varepsilon$. Then:
\begin{align*}
    \Pb_n(L = k) &= \Pb_n(\Lambda_k) - \sum_{j=1}^{\cl{n/2}-k}
        \Pb_n(\Lambda_{k+j}) \Pb_n(A_{k+j} = k) \\
  &= \begin{aligned}[t] \Pb_n(\Lambda_k) &- \sum_{j=1}^J I_{1/2}(k+j-1, \alpha)
        \frac{\alpha k\rf{j-1}}{(\alpha+k-1)\rf{j}}
        \lr*{1 + O\lr*{\frac{(k+j)^2}{\sqrt{n}}}} \\
      &- \sum_{j={J+1}}^{\cl{n/2}-k} \Pb_n(\Lambda_{k+j})
        \frac{\alpha k\rf{j-1}}{(\alpha+k-1)\rf{j}}, \end{aligned}
\end{align*}
As in the proof of Theorem~\ref{thm:P(D = h)}, which dealt with the depth of the
centroid, the upper bound for $\Pb_n(\Lambda_k)$ given in
equation~\eqref{eqn:P_n(Lambda_k(1/2)) bound} implies that the sum over larger
labels vanishes as $n$ grows:
\begin{align*}
    \sum_{j=J+1}^{\cl{n/2}-k} \Pb_n(\Lambda_{k+j}) \frac{\alpha k\rf{j-1}}
        {(\alpha+k-1)\rf{j}}
  &\le \sum_{j=J+1}^{\cl{n/2}-k} \frac{\alpha\rf{k+j-1}}{(k+j-1)!}
        \frac{\alpha k\rf{j-1}}{(\alpha+k-1)\rf{j}} 2^{-(k+j-2)} \\
  &= \frac{\alpha\rf{k-1}}{(k-1)!} \sum_{j=J+1}^{\cl{n/2}-k}
        \frac{\alpha}{k+j-1} 2^{-(k+j-2)} \\
  &\le \frac{\alpha\rf{k-1}}{(k-1)!} \frac{\alpha n}{\fl{n^{1/4-\varepsilon}}}
        2^{-\fl{n^{1/4-\varepsilon}}}
  \tendsto{n\to\infty} 0.
\end{align*}
Also, the extension of the first sum to an infinite one is permissible, since:
\begin{align*}
    \sum_{j>J} I_{1/2}(k+j-1, \alpha)
        \frac{\alpha k\rf{j-1}}{(\alpha+k-1)\rf{j}}
  &= \sum_{j>J} \frac{\alpha\Gamma(\alpha+k-1)}{\Gamma(\alpha)\Gamma(k)}
        \int_0^{1/2} t^{k+j-2} (1-t)^{\alpha-1}\, dt \\
  &\le \alpha\binom{\alpha+k-2}{k-1} \sum_{j>J} 2^{-(k+j-1)} \\
  &= O\lr*{2^{-(k+J-1)}} \tendsto{n\to\infty} 0.
\end{align*}
Letting $n \to \infty$, and assuming $k > 1$, we obtain the limiting probability
$\Pb(\mc{L} = k)$:
\begin{equation}\label{eqn:P(L = k) calculation}
\begin{aligned}
    \Pb(\mc{L} = k) &= \lim_{n\to\infty} \Pb_n(L = k) \\
  &= I_{1/2}(k-1, \alpha) - \sum_{j\ge1} I_{1/2}(k+j-1, \alpha)
        \frac{\alpha k\rf{j-1}}{(\alpha+k-1)\rf{j}} \\
  &= I_{1/2}(k-1, \alpha) - \alpha \frac{\Gamma(\alpha+k-1)}
        {\Gamma(k)\Gamma(\alpha)} \int_0^{1/2} (1-t)^{\alpha-1}
        \sum_{j\ge1} t^{k+j-2}\, dt \\
  &= I_{1/2}(k-1, \alpha) - \frac{\alpha}{\alpha-1}
        \Beta(k, \alpha-1)^{-1} \int_0^{1/2} t^{k-1} (1-t)^{\alpha-2}\, dt \\
  &= I_{1/2}(k-1, \alpha) - \frac{\alpha}{\alpha-1} I_{1/2}(k, \alpha-1).
\end{aligned}
\end{equation}
When $k = 1$, the first incomplete beta function is replaced by $1$. The
consolidated form given in equation~\eqref{eqn:P(L = k)} is due to the following
property of the incomplete beta function:
\begin{equation*}
    I_x(a-1, b) - \frac{x^{a-1} (1-x)^b}{(a-1) \Beta(a-1, b)}
  = I_x(a, b) = I_x(a, b-1) + \frac{x^a (1-x)^{b-1}}{(b-1) \Beta(a, b-1)},
\end{equation*}
\end{proof}

Repeating~\eqref{eqn:P(L = k) calculation} for $\alpha = 1$ yields:

\begin{corollary}[\citet{moon02centroid}]\label{cor:P(L = k)}
For recursive trees:
\begin{equation*} \lim_{n\to\infty} \Pb_n(L = k) = 2^{-(k-1)}
  - \sum_{j\ge k} \frac{2^{-j}}j. \end{equation*}
\end{corollary}

In particular (for recursive trees), $\lim_n \Pb_n(L = 1) = 1 - \ln2$.

\subsection{Moments of the label distribution} 

Just as we did when dealing with the depth of the centroid, we can apply
Lebesgue's dominated convergence theorem to prove convergence of the moments of
$\Pb_n(L = k)$ to those of $\Pb(\mc{L} = k)$, and then derive their limits using
the more convenient form of the limiting distribution. In the present case of
the centroid's label, however, the proof of convergence is almost immediate.

\begin{lemma}
The moments of the distribution of the centroid's label $L(\mc{T}_n)$ converge
to those of $\mc{L}$, i.e.:
\begin{equation*} \lim_{n\to\infty} \Ex_n(L^m) = \Ex(\mc{L}^m). \end{equation*}
\end{lemma}

\begin{proof}
The line of argument is the same as that which was used for
Lemma~\ref{lem:E(D) convergence}: we must find a uniform bound~$g(k)$ for
$k\ff{m} \Pb_n(L = k)$ such that $\sum_k k\ff{m}\, g(k)$ converges. Once again
by equation~\eqref{eqn:P_n(Lambda_k(1/2)) bound}:
\begin{align*}
    \sum_{k\ge1} k\ff{m} \Pb_n(L = k)
  &\le \sum_{k\ge m} k\ff{m} \Pb_n(\Lambda_k) \\
  &\le \sum_{k\ge m} \frac{k\ff{m}\, \alpha\rf{k-1}}{(k-1)!} 2^{-(k-2)} \\
  &= 2^{-(m-2)}\sum_{k\ge m-1}\binom{\alpha + k-1}{k} (k+1)\ff{m}\,
        2^{-(k-m+1)} \\
  &= 2^{-(m-2)}\frac{d^m}{du^m}\lrb*{u(1-u)^{-\alpha}}_{u=1/2} < \infty.
        \qedhere
\end{align*}
\end{proof}

\begin{theorem}\label{thm:E(L_star)}
The limit of the $m$th factorial moment of the centroid's label $L(\mc{T}_n)$ is
given by:
\begin{equation*} \Ex(\mc{L}\ff{m})
  = \frac{4m^2 + 2\alpha m + \alpha - 2}{m+1} \alpha\rf{m-1}. \end{equation*}
In particular, the limits of its mean and variance are:
\begin{gather*}
    \Ex(\mc{L}) = 1 + \frac32 \alpha, \\
    \Va(\mc{L}) = -\frac{7}{12} \alpha^2 + \frac{19}{6} \alpha.
\end{gather*}
\end{theorem}

\begin{proof}
The factorial moments of the limiting distribution can be computed directly
using equation~\eqref{eqn:P(L = k)}:
\begin{align*}
    \sum_{k\ge1} k\ff{m} \Pb(\mc{L} = k) &= -\frac{1}{\alpha-1}
        \sum_{k\ge1} k\ff{m}\, I_{1/2}(k, \alpha) \\
  &\qquad+ \lr*{1 + \frac{\alpha}{\alpha-1}} \sum_{k\ge1}
        \binom{\alpha+k-2}{k-1} k\ff{m}\, 2^{-(\alpha+k-1)} \\
  &= -\frac{1}{\alpha-1} \int_0^{1/2} (1-t)^{\alpha-1} \sum_{k\ge1}
        \binom{\alpha+k-1}{k-1} \alpha k\ff{m}\, t^{k-1}\, dt \\
  &\qquad+ \lr*{1 + \frac{\alpha}{\alpha-1}} \sum_{k\ge1}
        \binom{\alpha+k-2}{k-1} k\ff{m}\, 2^{-(\alpha+k-1)} \\
  &= -\frac{\alpha}{\alpha-1} \int_0^{1/2} (1-t)^{\alpha-1} t^{m-1}
        \frac{d^m}{dt^m} \lrb*{t(1-t)^{-(\alpha+1)}}\, dt \\
  &\qquad+ \lr*{1 + \frac{\alpha}{\alpha-1}} 2^{-(\alpha+m-1)}
        \frac{d^m}{dt^m} \lrb*{t(1-t)^{-\alpha}}_{t=1/2} \\
  &= -\frac{\alpha}{\alpha-1} \int_0^{1/2} \sum_{i=0}^1 \binom{m}{i}
        (\alpha+1)\rf{m-i}\, t^{m-i} (1-t)^{-(m+2-i)}\, dt \\
  &\qquad+ \lr*{1 + \frac{\alpha}{\alpha-1}} \lr*{\alpha\rf{m} +
        m\alpha\rf{m-1}}.
\end{align*}
Noting that the integrals within the sum are all of a common, solvable form:
\begin{equation*} \int_0^{1/2} t^m (1-t)^{-(m+2)}\, dt = \frac{1}{m+1}
    \lrb*{\lr*{\frac{t}{1-t}}^{m+1}}_{t=0}^{1/2} = \frac{1}{m+1}, \end{equation*}
the $m$th factorial moment reduces to:
\begin{align*}
    \sum_{k\ge1} k\ff{m} \Pb(\mc{L} = k) &= -\frac{\alpha}{\alpha-1}
        \lr*{\frac{(\alpha+1)\rf{m}}{m+1} + \frac{m(\alpha+1)\rf{m-1}}{m}} \\
  &+ \lr*{1 + \frac{\alpha}{\alpha-1}} \lr*{\alpha\rf{m} + m \alpha\rf{m-1}} \\
  &= \lr*{1 + \frac{\alpha}{\alpha-1}} m\alpha\rf{m-1} + 2\alpha\rf{m}
      - \frac{1}{\alpha-1} \frac{\alpha\rf{m+1}}{m + 1} \\
  &= \frac{4m^2 + 2\alpha m + \alpha - 2}{m+1} \alpha\rf{m-1}. \qedhere
\end{align*}
\end{proof}

Once again, the fact that the expected label of the centroid in a random
recursive tree tends to $5/2$ was first proved by \citet{moon02centroid}. And
lastly, but not unexpectedly, it follows from Theorem~\ref{thm:E(L_star)} that
binary increasing trees ($\alpha = 2$) lead to the greatest eventual mean and
variance.

\section{The size of the centroid's root branch}
\label{sec:root branch}

Our third and final set of results involving the centroid of an increasing tree
involve its ancestral branch. In a way this is the most interesting of the
centroid's branches, because its descendent branches behave mostly (in
particular, their number and sizes do) like those of the root branches of a
random increasing tree---albeit under the extra condition that no one branch
contains more than $\fl{n/2}$ nodes.

It is also worth noting that these results are interesting for another reason:
they can be contrasted with the case of simply generated trees. We have already
mentioned that almost all simply generated trees of size~$n$ have three large
centroid branches that together contain most of the tree's nodes, and in fact
\citet{meir02centroid} have proved (among other things) that the size of the
centroid's ancestral branch, divided by~$n$, tends to $\sqrt{2} - 1 \approx
0.414$ as $n \to \infty$ (independently of the specific family of simply
generated trees). Our main goal in this section is an analogue of this result
for increasing trees, however we will phrase it (analogously) in terms of the
size of the subtree rooted at the centroid. This sort of comparison was not
possible for the results of the previous two sections, simply due to the fact
that the depth and label (where applicable) of the centroid in a simply
generated tree are relatively uninformative, because roots or specifically
labelled nodes in simply generated trees are, for the most part, no different
from randomly selected nodes.

Let $S(T)$ denote the size of the subtree containing the centroid and all of its
descendent branches, and $\Pb_n(S = m = \fl{\theta n})$ the relevant probability
distribution. Since the ancestral branch contains at most $\fl{n/2}$ nodes, the
ranges of~$m$ and~$\theta$ are $\{\cl{n/2}, \dotsc, n\}$ and $[1/2, 1]$
respectively, with $m = n$ characterising the case in which the root and
centroid coincide (for which Theorem~\ref{thm:P(L = k)} already provides a
limiting probability).

\subsection{A preliminary equation} 

The event that the centroid's subtree is made up of~$m$ nodes (where $n/2 \le m
< n$) can be decomposed into a pair of simpler events: firstly, that the tree
contains a subtree of size~$m$ (there can be at most one); and secondly, given
the presence of such a subtree, that its root is the centroid. This second event
can be stated more explicitly as the case, in a tree of size~$m$, that the
root is the only node with at least $\fl{n/2}$ descendants.

It is here that we will make use of $\Pb_n(\Lambda_k(\sigma))$, which was
introduced in Section~\ref{sec:depth path} as a generalisation of the `path'
probability $\Pb_n(\Lambda_k(1/2))$. Let $X_m(T)$ mark the existence of a
subtree of size~$m$ in a tree~$T$, and let $F_j$ be the label of node~$j$'s
parent, so that $F_j = 1$ characterises the root's children. The probability
that the centroid's subtree contains exactly $m$ nodes can be expressed as:
\begin{align}\label{eqn:P_n(S = m)}
    \Pb_n(S = m) &= \Pb_n(X_m) \lr*{1 - \Pb_m\lr*{\ts \bigcup_{j\ge2}
        \Lambda_j\lr*{\frac{n}{2m}}}} \nonumber \\
  &= \Pb_n(X_m) \lr[\Bigg]{1 - \sum_{j=2}^m \Pb_m\lr*{F_j = 1 \cap
        \Lambda_j\lr*{\ts \frac{n}{2m}}}} \nonumber \\
  &= \Pb_n(X_m) \cdot \lr*{1 - A_m\lr*{\ts \frac{n}{2m}}},
\end{align}
where $A_m\lr*{\ts \frac{n}{2m}}$ is used as an abbreviation for \begin{math}\sum_{j=2}^m \Pb_m\lr*{F_j = 1 \cap
        \Lambda_j\lr*{\ts \frac{n}{2m}}}\end{math}.
The probabilities that appear within the sum
refer to disjoint events, since at most one of the subtree's root branches can
contain $\fl{n/2} + 1$ nodes. We note, as we did for
equation~\eqref{eqn:P_n(Lambda cap A)}, that the size of the subtree rooted at
$j$ is independent of the node~$j$ was attached to, so that:
\begin{equation*} \Pb_m\lr*{F_j = 1 \cap \Lambda_j\lr*{\ts \frac{n}{2m}}}
  = \Pb(F_j = 1) \Pb_m\lr*{\Lambda_j\lr*{\ts \frac{n}{2m}}}. \end{equation*}

Consider the first probability $\Pb_n(X_m)$ in equation~\eqref{eqn:P_n(S = m)},
of the event that a tree of~$n$ nodes contains a subtree of size~$m$. Since
there can be at most one, this probability can be rephrased as the expected
number of such subtrees. This problem in turn is known as the
\emph{subtree size profile} of a tree, and has recently been studied for various
families of increasing trees. In fact, the expected proportion of nodes with
$m-1$ descendants (each forming a rooted subtree of size~$m$) has already been
given explicitly for the most interesting families, see~\citep[Section~3 and
Theorem~4.1]{fuchs12limit}. Letting $U_m = U_m(T)$ denote the number of subtrees of
size~$m$ in a random tree, we perform the derivation in a more general way here.

\begin{lemma}\label{lem:E_n(U_m)}
For $1 \le m < n$, the expected number of subtrees of size~$m$ in a random very
simple increasing tree of size~$n$ is given by:
\begin{equation*} \Ex_n(U_m) = \frac{\alpha(\alpha+n-1)}{(\alpha+m)(\alpha+m-1)}. \end{equation*}
\end{lemma}

\begin{proof}
Firstly, note that \begin{math}\Ex_n(U_m) = \sum_l \Pb_n(S_l = m)\end{math}, where~$S_l$ is the size
of the subtree rooted at~$l$. Now $\Pb_n(S_l = m)$ is the probabilty that node $l$
has $m-1$ descendants, which was derived in Section~\ref{sec:depth path} (see in
particular~\eqref{eqn:P_n(Lambda_k(sigma)) sum}) to be
\begin{equation*}
\Pb_n(S_l = m) = \binom{\alpha+m-2}{m-1} \binom{n-m-1}{l-2} \Bigg/ \binom{\alpha+n-2}{n-l},
\end{equation*}
which can be rewitten as
\begin{equation*}
\Pb_n(S_l = m) = \alpha\Beta(n-m, \alpha+m-1) \binom{\alpha+l-2}{l-2} \binom{n-l}{m-1}.
\end{equation*}

Summing over possible labels (the root is omitted since $m < n$) yields:
\begin{align*}
    \Ex_n(U_m) &= \sum_{l=2}^{n-m+1} \Pb_n(S_l = m) \\
  &= \alpha \Beta(n-m, \alpha+m-1)
        \sum_{l=2}^{n-m+1} \binom{\alpha+l-2}{l-2} \binom{n-l}{m-1} \\
  &= \alpha \Beta(n-m, \alpha+m-1) \binom{\alpha+n-1}{n-m-1}.
\end{align*}
in which the final step is due to the Chu-Vandermonde identity, once the
numerators of the binomial coefficients have been converted to constants:
\begin{equation*} \sum_{l=0}^{n-m-1} \binom{\alpha + l}{l} \binom{n-l-2}{m-1}
  = (-1)^{n-m-1} \sum_{l=0}^{n-m-1} \binom{-\alpha-1}{l}
    \binom{-m}{n-m-1-l}. \end{equation*}
The stated result is obtained after simplifying.
\end{proof}

\begin{corollary}\label{cor:P_n(X_m)}
For $n/2 \le m < n$, the probability that a tree of size~$n$ contains a subtree
of size~$m$ is:
\begin{equation*} \Pb_n(X_m) = \frac{\alpha(\alpha+n-1)}{(\alpha+m)(\alpha+m-1)}. \end{equation*}
\end{corollary}

\subsection{The probability that the root of a subtree is the centroid} 

The second probability in equation~\eqref{eqn:P_n(S = m)}, denoted by $1 -
A_m(n/(2m))$, accounts for the cases in which the root of a subtree of size~$m$
is the centroid of the entire tree, where $m/n = \theta$ for some fixed
$\theta$. By the aforementioned independence argument, we have:
\begin{equation}\label{eqn:A_m(n/2m)}
    A_m\lr*{\ts \frac{n}{2m}}
  = \sum_{j=2}^m \Pb(F_j = 1) \Pb_m\lr*{\Lambda_j\lr*{\ts \frac{n}{2m}}},
\end{equation}
in which both of the terms contained within the sum are manageable---the first
by Lemma~\ref{lem:P(A = k)}:
\begin{equation*} \Pb(F_j = 1) = \frac{\alpha (j-2)!}{\alpha\rf{j-1}}
  = \alpha \Beta(j-1, \alpha), \end{equation*}
and the second due to Theorem~\ref{thm:P_n(Lambda_k(sigma))} and
Lemma~\ref{lem:P_n(Lambda_k(sigma)) bound}, which provide an asymptotic form and
an upper bound respectively.

\begin{lemma}\label{lem:A_m(n/2m)}
For $n/2 \le m < n$ in a tree of size~$n$, the probability that the root of a
subtree of size~$m$ is \emph{not} the centroid of the tree satisfies, for all
$0 < \varepsilon < 1/2$:
\begin{equation*} A_m\lr*{\ts \frac{n}{2m}} = \frac{\alpha}{\alpha-1}
    \lr*{1 - \lr*{\ts \frac{n}{2m}}^{\alpha-1}} + O\lr*{n^{-\varepsilon}}. \end{equation*}
\end{lemma}

\begin{proof}
The sum given in equation~\eqref{eqn:A_m(n/2m)} can be split at a value small
enough for Theorem~\ref{thm:P_n(Lambda_k(sigma))} to be applied, say $J =
\fl{m^{1/4-\varepsilon/2}}$ so that $J^2/\sqrt{m} = O(m^{-\varepsilon})$:
\begin{align*}
    A_m\lr*{\ts \frac{n}{2m}} &= \sum_{j=2}^J \Pb(F_j = 1) I_{1-\frac{n}{2m}}
      (j-1, \alpha) \lr*{1 + O\lr*{\frac{j^2}{\sqrt{m}}}} \\
  &+ \sum_{j=J+1}^{m} \Pb(F_j = 1) \Pb_m\lr*{\Lambda_j\lr*{\ts \frac{n}{2m}}}.
\end{align*}
Applying the bound of Lemma~\ref{lem:P_n(Lambda_k(sigma)) bound} affirms that
the second sum is small for large values of~$J$:
\begin{align*}
    \sum_{j=J+1}^m \Pb(F_j = 1) \Pb_m\lr*{\Lambda_j\lr*{\ts \frac{n}{2m}}}
  &\le 6\alpha \frac{m}{n} \sum_{j=J+1}^m \Beta(j-1, \alpha)
        \frac{\alpha\rf{j-1}}{(j-1)!} \lr*{1 - \frac{n}{2m}}^{j-1} \\
  &\le 6\alpha \frac{m}{n} \sum_{j=J+1}^m \lr*{1 - \frac{n}{2m}}^{j-1} \\
  &< 12\alpha \lr*{\frac{m}{n}}^2 \lr*{1 - \frac{n}{2m}}^J
  \tendsto{m\to\infty} 0,
\end{align*}
Similarly, extending the first sum to an infinite one has an effect that
vanishes as~$m$ and~$n$ grow:
\begin{align*}
    \sum_{j>J} \Pb(F_j = 1) I_{1-\frac{n}{2m}}(j-1, \alpha)
  &= \alpha \sum_{j>J} \int_0^{1-\frac{n}{2m}} t^{j-2}
        (1-t)^{\alpha-1}\, dt \\
  &\le \alpha \sum_{j>J} \lr*{1 - \frac{n}{2m}}^{j-1}
  \tendsto{m\to\infty} 0.
\end{align*}
Combined, these two substitutions give the asymptotic behaviour of
$A_m(\frac{n}{2m})$:
\begin{align*}
    A_m\lr*{\ts \frac{n}{2m}} &= \alpha \sum_{j=2}^J \int_0^{1-\frac{n}{2m}}
        t^{j-2} (1-t)^{\alpha-1}\, dt \lr*{1 + O\lr*{\frac{J^2}{\sqrt{m}}}}
      + O\lr*{\lr*{1 - \frac{n}{2m}}^J} \\
  &= \alpha \int_0^{1-\frac{n}{2m}} (1-t)^{\alpha-2}\, dt
        \lr*{1 + O\lr*{m^{-\varepsilon}}} \\
  &= \frac{\alpha}{\alpha-1} \lr*{1 - \lr*{\frac{n}{2m}}^{\alpha-1}}
      + O\lr*{n^{-\varepsilon}},
\end{align*}
which can be compared to the case $m = n$ as given in
Theorem~\ref{thm:P(L = k)}.
\end{proof}

When dealing with recursive trees, the final step is slightly different,
resulting in:

\begin{corollary}
For recursive trees:
  \begin{equation*} A_m\lr*{\ts \frac{n}{2m}} \sim \ln\lr*{\ts \frac{2m}{n}}. \end{equation*}
\end{corollary}

\subsection{The distribution of the size of the centroid's subtree} 

Now that $\Pb_n(X_m)$ is known explicitly (Corollary~\eqref{cor:P_n(X_m)}), and
$A_m(\frac{n}{2m})$ asymptotically, we are ready to derive an expression for the
distribution of $S(T)$. In the light of equation~\eqref{eqn:P_n(S = m)}, which
states that:
\begin{equation*} \Pb_n(S = m) = \Pb_n(X_m) \cdot \lr*{1 - A_m\lr*{\ts \frac{n}{2m}}}, \end{equation*}
we have the main theorem of this section:

\begin{lemma}\label{lem:P_n(S = m)}
For $n/2 \le m < n$ and any $0 < \varepsilon < 1/2$, the probability that the
centroid has $m-1$ descendent nodes is given by:
\begin{equation*} \Pb_n(S = m) = \frac{4}{n} \frac{\alpha}{\alpha-1}
    \lr*{\alpha \lr*{\frac{n}{2m}}^{\alpha+1} - \lr*{\frac{n}{2m}}^2}
  + O\lr*{n^{-1-\varepsilon}}. \end{equation*}
\end{lemma}

\begin{proof}
The result is simply an application of Corollary~\eqref{cor:P_n(X_m)} and
Lemma~\ref{lem:A_m(n/2m)} to the above expression:
\begin{align*}
    \Pb_n(S = m)
  &= \lr*{1 - \frac{\alpha}{\alpha-1} \lr*{1 - \lr*{\frac{n}{2m}}^{\alpha-1}}
      + O\lr*{n^{-\varepsilon}}} \frac{\alpha (\alpha+n-1)} {(\alpha+m)
        (\alpha+m-1)} \\
  &= \lr*{1 - \frac{\alpha}{\alpha-1} \lr*{1 - \lr*{\frac{n}{2m}}^{\alpha-1}}}
        \frac{\alpha}{n} \lr*{\frac{n}{m}}^2 + O\lr*{n^{-1-\varepsilon}} \\
  &= \frac{4}{n} \frac{\alpha}{\alpha-1}
          \lr*{\alpha \lr*{\frac{n}{2m}}^{\alpha+1}
        - \lr*{\frac{n}{2m}}^2} + O\lr*{n^{-1-\varepsilon}}. \qedhere
\end{align*}
\end{proof}

Combined with the special case $m = n$ (Theorem~\ref{thm:P(L = k)}), this
asymptotically describes the size of the centroid's subtree, and thus the
size $n-m$ of its root branch as well.

As with the distributions of the centroid's depth and label, we can also show
convergence to a limiting distribution; however in this case, although the
finite probability distributions are discrete, the limiting distribution is a
mixture of a continuous distribution with support $[1/2, 1)$ and a point measure
at~$1$. The notion of convergence is also slightly different, in that it is
weaker than those of the previous two sections.

\begin{theorem}\label{thm:P(S = theta)}
The proportion $S(\mc{T}_n)/n$ of nodes accounted for by the subtree consisting
of the centroid and all of its descendants in a random tree of size~$n$
converges in distribution to the random variable $\mc{S}$, defined on $[1/2, 1)$
by the density:
\begin{equation*} f(\theta) = 4 \frac{\alpha}{\alpha-1}
    \lr*{\alpha (2\theta)^{-(\alpha+1)} - (2\theta)^{-2}}, \end{equation*}
and at the boundary $\theta = 1$ by the point measure:
\begin{equation*} \Pb(\mc{S} = 1) = 1 - \frac{\alpha}{\alpha-1} \lr*{1 - 2^{-(\alpha-1)}}. \end{equation*}
\end{theorem}

\begin{proof}
Consider the cumulative distribution function arising from
Lemma~\ref{lem:P_n(S = m)}:
\begin{align*}
    \Pb_n(S \le \sigma n) &= \frac{4}{n} \frac{\alpha}{\alpha-1}
        \sum_{m=\cl{n/2}}^{\fl{\sigma n}} \lr*{\alpha
        \lr*{\frac{n}{2m}}^{\alpha+1} - \lr*{\frac{n}{2m}}^2}
      + O\lr*{n^{-\varepsilon}} \\*
  &= 4 \frac{\alpha}{\alpha-1} \int_{1/2}^\sigma
        \lr*{\alpha 2\theta)^{-(\alpha+1)} - (2\theta)^{-2}}\, d\theta
      + O\lr*{n^{-\varepsilon}}.
\end{align*}
Note that the error term---which traces back to
Theorem~\ref{thm:P_n(Lambda_k(sigma))} via Lemma~\ref{lem:A_m(n/2m)}---is
uniform in~$\sigma$ over subsets of the form $[1/2, 1 - \delta)$. And since each
point of continuity (there is discontinuity at~$1$) is contained in such a
subset, this makes explicit (for $\sigma < 1$) the convergence of $\Pb_n(S \le
\sigma n)$ to the continuous distribution with the stated density. The point
measure simply corresponds to $\Pb(\mc{L} = 1)$. \end{proof}

Once again, the result for recursive trees differs slightly:
\begin{corollary}\label{cor:P(S = m)}
For recursive trees:
\begin{equation*} f(\theta) = \frac{1 - \ln(2\theta)}{\theta^2} \text{ on $[\tfrac12,1)$}\quad\text{and}\quad
    \Pb(\mc{S} = 1) = 1-\ln2. \end{equation*}
\end{corollary}

\subsection{Moments of the subtree's size distribution} 

Finally, we can detail the limiting behaviour of the moments of $S(T)$, and in
particular, the expected size of the centroid's subtree. This is simply
mechanical, since our random variables have bounded support, so that convergence
in distribution implies convergence of moments.

\begin{theorem}
The moments of the distribution of the proportion $S(\mc{T}_n)/n$ of the tree
accounted for by the centroid and its descendants converge to those of $\mc{S}$,
i.e.:
\begin{equation*} \lim_{n\to\infty} \Ex_n((S/n)^r) = \Ex(\mc{S}^r). \end{equation*}
The limit of the $r$th moment satisfies:
\begin{equation*} \Ex(\mc{S}^r) = \Pb(\mc{S} = 1) + \frac{\alpha}{\alpha-1}
    \lr*{\frac{\alpha}{\alpha-r} \lr*{2^{-(r-1)} - 2^{-(\alpha-1)}}
  - \frac{1}{r-1} \lr*{1 - 2^{-(r-1)}}}. \end{equation*}
In particular, the limit of its mean is:
\begin{equation*} \Ex(\mc{S}) = 1 + \frac{\alpha}{(\alpha-1)^2}
    \lr*{1 - 2^{-(\alpha-1)} - (\alpha-1) \ln2}. \end{equation*}
\end{theorem}

\begin{proof}
For $\alpha \notin \{1, r\}$ and $r \in \mb{Z}_{>0}$, and with $f(\theta)$ as in
Theorem~\ref{thm:P(S = theta)}, we have:
\begin{align*}
  &\Ex(\mc{S}^r)
  = \Pb(\mc{S} = 1) + \int_{1/2}^1 \theta^r f(\theta)\, d\theta \\
  &= \Pb(\mc{S} = 1) + 2^{-(r-1)} \frac{\alpha}{\alpha-1}
        \int_1^2 \mu^r \lr*{\alpha \mu^{-(\alpha+1)} - \mu^{-2}}\, d\mu \\
  &= \Pb(\mc{S} = 1) + 2^{-(r-1)} \frac{\alpha}{\alpha-1}
        \lrb*{-\frac{\alpha}{\alpha-r} \mu^{r-\alpha} - \frac{1}{r-1}
        \mu^{r-1}}_1^2 \\
  &= \Pb(\mc{S} = 1) + \frac{\alpha}{\alpha-1} \lr*{\frac{\alpha}{\alpha-r}
        \lr*{2^{-(r-1)} - 2^{-(\alpha-1)}} - \frac{1}{r-1}
        \lr*{1 - 2^{-(r-1)}}}. \qedhere
\end{align*}
\end{proof}

For plane-oriented and binary increasing trees ($\alpha = 1/2$ and
$\alpha = 2$), this yields means of $3 - 2\sqrt{2} + \ln2 \approx 0.86$ and
$2 - 2\ln2 \approx 0.61$ respectively. It is worth noting that the proof's
requirement that $\alpha \ne r$ is weak for two reasons: because the standard
definition of very simple increasing trees deals with the range $0 < \alpha \le
2$, and because singular cases such as these can be seen as limits of the above
function, or, in the case of recursive trees, be derived directly from
Corollaries~\ref{cor:P(L = k)} and~\ref{cor:P(S = m)}. For instance:
\begin{gather*}
    \Ex(\mc{S})|_{\alpha=1} = 1 - \frac12 (\ln2)^2 \approx 0.76, \\
    \Ex(\mc{S}^r)|_{\alpha=1} = 1 + \frac{r}{(r-1)^2} \lr*{1 - 2^{-(r-1)}}
      - \frac{r}{r-1} \ln2, \\
    \Ex\lr{\mc{S}^2}|_{\alpha=2} = 2\ln2 - 1.
\end{gather*}
The limit of the mean in the case of recursive trees was given by
\citet{moon02centroid}. The asymptotic variances for plane-oriented, recursive,
and binary increasing trees are approximately~$0.03$, $0.04$, and~$0.01$
respectively.

\section{Concluding remarks} 
\label{sec:conclusion}

The behaviour of the (nearest) centroid in a large, random very simple
increasing tree is now reasonably clear, and in fact quite consistent across the
entire subclass: one expects the centroid to lie, on average, within two edges
of the root (for the usual case $0 < \alpha \le 2$), and its root branch to
account for a significant portion of the entire tree.

That being said, there are still a number of related questions that could be
raised and investigated: for example, we might consider other parameters of the
centroid---in the simplest case, its degree---or ask to what extent the above
behaviour generalises to the entire class of increasing trees, which do not
necessarily satisfy the properties of Lemma~\ref{lem:panholzer07}.

Perhaps more interestingly, one could attempt to characterise the distribution
of the average distance from a node to the other nodes in a tree---the closeness
centrality---in a way similar to that which has been done for betweenness
centrality~\citep{durant17betweenness}; or, as an alternative definition of a
tree's most `central' node, consider its central points (i.e., nodes with
minimal eccentricity). Finally, there are classes of random trees other than
simply generated and increasing trees that we have not mentioned at all
here---the most notable being the various families of search trees
\cite[Section~1.4]{drmota09random}. One would expect (or at least hope) that
several of these problems will also be amenable to the tools and methods of
analytic combinatorics.


\end{document}